\documentclass[11pt]{article}%
\usepackage{amsmath}%
\usepackage{amsfonts}%
\usepackage{amssymb}%
\usepackage{graphicx}

\newcommand{\Cb}{{\bf C}}

\newcommand{\Vb}{{\bf V}}

\newcommand{\F}{{\mathbb F}}

\newcommand{\Z}{{\mathbb Z}}

\newcommand{\Q}{{\mathbb Q}}
\newcommand{\Pp}{{\mathbb P}}
\newcommand{\CC}{{\mathbb C}}

\newcommand{\A}{{\mathcal A}}
\newcommand{\B}{{\mathcal B}}
\newcommand{\Cc}{{\mathcal C}}

\newcommand{\Ee}{{\mathcal E}}

\newcommand{\KK}{{\mathcal K}}

\newcommand{\V}{{\mathcal V}}
\newcommand{\W}{{\mathcal W}}

\newcommand{\Hm}{{\text{Hom}}}

\newcommand{\alb}{{\text{Alb}}}
\newcommand{\aut}{{\text{Aut}}}
\newcommand{\enn}{{\text{End}}}
\newcommand{\prym}{{\text{Prym}}}
\newcommand{\rk}{{\text{rk}}}

\newtheorem{thm}{Theorem}[section]

\newtheorem{lema}[thm]{Lemma}

\newtheorem{prop}[thm]{Proposition}

\newtheorem{deff}[thm]{Definition}
\newtheorem{rem}[thm]{Remark}

\numberwithin{equation}{section}
\newenvironment{proof}[1][Proof]{\textbf{#1.} }{\ \rule{0.5em}{0.5em}}
\begin{document}

\title{The rational points on certain Abelian varieties over function fields}
\author{Sajad Salami, \\
Institute of Mathematics and Statistics \\
State University of Rio de Janeiro, Brazil}
\date{}
\maketitle
\begin{abstract}
In this paper, we consider    Abelian varieties over function fields that  arise
as twists of Abelian varieties by  cyclic covers of  irreducible quasi-projective varieties. 
Then, in terms of Prym varieties associated to the cyclic  covers, we  prove a structure theorem   on their  Mordell-Weil group. 
Our results give an explicit method for construction of   elliptic curves, hyper- and super-elliptic Jacobians that have large  ranks over function fields of certain varieties.
\end{abstract}

{\bf Subjclass 2010: }{Primary 11G10; Secondary  14H40} %

{\bf keywords}: Mordell-Weil rank, Abelian variety, function field of varieties,  Prym varieties associated to the cyclic covers, Jacobian variety.

\section{Introduction and main results}

Let $\A$ be an Abelian variety defined  over an arbitrary global  field $k$. By the Mordell-Weil Theorem,  the set  $\A (k)$
of $k$-rational points on $A$ is a  finitely generated abelain group  \cite{slm3}. In other words,  one has   $\A (k)\cong  \A (k)_{tors} \oplus \Z^r$ where $\A (k)_{tors}$ is   
a finite subgroup of $\A (k)$ that is  called  the torsion subgroup; and  $r$ is a  non-negative number 
$r$  which is called the (Mordell-Weil) rank of $\A$ over $k$ and is denoted  by $\rk (\A (k))$.
It is a  mysterious  quantity  associated to  an Abelian variety.  
Finding Abelian varieties  with large ranks    is one of the most challenging  problems in  Arithmetic   and Diophantine  geometry. 

For example,  when $\A$ is an elliptic curve defined over $k=\Q$   a folklore conjecture  asserts that  the 
rank of an elliptic curve over $\Q$ can be arbitrary large \cite{slm1}. This conjecture should now be regarded as being in serious doubt by the results of J. Park and  et. al., in \cite{poonen}. In short, they predict that the 
number of elliptic curves over $\Q$ with rank $\geq 21$ is finite. However, in 2006, Elkies showed the 
existence of   an elliptic  curve over $\Q$ with  $ 28$ independent generators.
In the case of elliptic curves over quadratic number  fields $k=\Q(\sqrt{d})$
the largest known rank is $30$ with $d=-3$, which has been found by    F. Najman.
To see the equation of these curves and more information on the  high rank elliptic
curves over rational numbers and quadratic number fields, we refer the reader to \cite{duje1}.

In contrast, for each prime $p$,  there are known explicit elliptic curves over $k=\F_p(t)$ with arbitrary large rank 
\cite{tate-shaf,ulmer1}. In the case $k=\CC (t)$,   it has been proved that for a very general elliptic curve $E$ over $k$ 
with height $d \geq 3$ and  every finite rational extension $k'$  of $k$ the Mordell-Weil  group $E(k')$
is a trivial group, see \cite{ulmer2}.

In this paper, we  generalize the main result of  Hazama in \cite{haz3} 
to an arbitrary cyclic $s$-covers of irreducible quasi-projective varieties for any integer  $s \geq 2$. 
We fix a global field  $k$  of characteristic $\geq 0$ not dividing  $s$, so that
it contains an $s$-th root of unity, which is  denoted by $\zeta$.
Let us denote by $\A [s](k)$ the subgroup of  $k$-rational $s$-division points of $\A$.
We assume  that there exists an order $s$ automorphism $\sigma \in \aut(\A).$  
Given the irreducible quasi-projective varieties $\V$ and $\V'$ with function fields $\KK$ and $\KK'$, respectively, 
we denote by $\prym_{\V'/\V}$  the  Prym variety associated to the cyclic $s$-cover $\pi : \V' \rightarrow \V$, all defined over $k$. 
See Section \ref{Ehcc} for the definition and some properties of  $\prym_{\V'/\V}$.
Let  $G=\left\langle  \gamma \right\rangle$  be the order $s$ cyclic Galois group of the  extension $\KK'|\KK$.  
Denote by  $\A_a$  the twist of $\A$ with  the extension  $\KK'|\KK$, 
equivalently by the $1$-cocycle $a=(a_u) \in Z^1(G, \aut(\A))$ 
given by $a_{id}=id$ and $a_{\gamma^j}=\sigma^j$ for each $\gamma^j \in G$, see \cite{bor-ser, haz1}.
Then, we have  the following theorem which is the main result of this article.
\begin{thm}
	\label{main1} 
	Notation being as above, we  assume that there exists a  $k$-rational  point $v'_0 \in \V'(k)$.  
	Then we have an isomorphism of  Abelian groups:
	$$\A_a(\KK) \cong \Hm_{k} (\prym_{\V'/\V}, \A) \oplus \A [s](k).$$
	Moreover, if  $\prym_{\V'/\V}$ is $k$-isogenous to $ \A^n \times \B$ for some positive integer $n$ and
	and an  Abelian variety   $\B$  over $k$ with  $\dim(\B)=0$   or  $\dim(\B)> \dim (\A)$ and  no  irreducible components  $k$-isogenous to $\A$,  then 
	$\rk(\A_a(K))\geq n\cdot \rk (\enn_{k}(\A)).$
\end{thm}
As an application of this theorem, for given integers $2\leq s \leq r \leq n  $,
we   consider the  cyclic $s$-cover $\pi: \Cb_n \rightarrow \Vb_n$ 
where $\Cb_n$ is the product of $n$ copies of 
the curve $\Cc_{s, f}$ given by the affine equation $y^s=f(x)$ 
with $f(x) \in k[x]$ of degree $r$, and  $\Vb_n$ is   the quotient of $\Cb_n$ by a certain cyclic subgroup of order  $s$ of $\aut (\Cb_n)$, see Section \ref{rpavff}. Let   $\Cc_{s, f}^{\xi}$  be the twist  of $\Cc_{s, f}$  by the cyclic extension $L| K$, 
where $K=k(\Vb_n)$ and $L=k(\Cb_n)$. 
Denote by $J(\Cc_{s, f})$  the Jacobian variety of $\Cc_{s, f}$ and let $J(\Cc_{s,f})[s](k)$ 
to be its subgroup of $k$-rational $s$-division points. We have the following result.
\begin{thm}
	\label{main2}
	With the above notations and  assuming that there exists $k$-rational  point $c \in \Cc_{s,f}(k)$,   we have
	$$J(\Cc_{s,f}^\xi)(K) \cong  \big (\enn_k( J(\Cc_{s,f}))\big)^n \oplus J(\Cc_{s,f})[s](k),$$
	an isomorphism of  Abelian groups and hence,  
	$$\rk(J(\Cc_{s,f}^\xi)(K))\ge n\cdot \rk (\enn_{k}(J(\Cc_{s,f}))).$$
\end{thm}

The structure of this paper is as follows. In  section  \ref{Ehcc}, we investigate some of the properties of
the Prym varieties associated to the cyclic covers of quasi-projective varieties.    In section \ref{hazres},
we recall the main result of   Hazama from  \cite{ haz3} that we are going to extend in this paper. 
Then,  we prove  Theorems \ref{main1} and \ref{main2} in  sections \ref{mproof1} and 
\ref{rpavff}, respectively.

\section{The Prym variety associated to   the cyclic  covers}
\label{Ehcc}

The notion of Prym variety was introduced by Mumford  in \cite{mumf} and   has been extensively studied  
for double covers of curves in  \cite{beavi}. It has been generalized  for double covers of  irreducible quasi-projective 
varieties in \cite{haz3}. Here, we generalize this notion  to the case of  cyclic $s$-covers of 
varieties.
\begin{deff}
	For an integer $s\geq 2$, the {\it Prym variety} of  the cyclic  $s$-cover  $\pi : \V' \rightarrow \V$ of irreducible 
	quasi-projective varieties over $k$ is defined by  the quotient  Abelian variety 
	$$\prym_{\V'/\V}:=\frac{\alb(\V')}{\text{Im} (id+\tilde{\gamma} +\cdots +\tilde{\gamma}^{s-1})},$$
	where $\alb (\V')$ is the {\it Albanese variety} and $\tilde{\gamma}$ is the automorphism of  $\alb(\V')$
	induced by an   order $s$ automorphism $\gamma \in \aut(\V') $ defined over $k$.
\end{deff}

We note that if both of the varieties $\V$ and $\V'$ are curves, then this definition  is 
compatible with the  one  given in  \cite{ortega}, by the following lemma.
\begin{lema}
	\label{alblem}
	Given an integer $s\geq 2$, let   $\pi: \V' \rightarrow \V$ be a cyclic  $s$-cover of  irreducible quasi-projective varieties, 
	both as well as $\pi$ defined over $k$.
	Suppose that $\gamma \in \aut(\V')$ is an  automorphism of order $s$ defined over $k$. 
	Denote  by $\tilde{\gamma}$  the automorphism of the Albanese variety $\alb(\V')$ induced by $\gamma.$
	Then  there is a $k$-isogeny of Abelian varieties,
	$$\prym_{\V'/\V} \sim_k \ker(id + \tilde{\gamma} + \cdots + \tilde{\gamma}^{s-1} :  \alb(\V') \rightarrow \alb(\V'))^\circ,$$
	where $(*)^\circ$ means the connected component of its origin.
\end{lema}
\begin{proof}
	Let  $\A=\alb(\V')$ and denote its dimension by $m$. Given 
	$\gamma\in \aut(\A)$ of order $s$. Define
	$m_1:= \dim \ker (id - \gamma)^\circ,$ and $   m_2:= \dim \ker (id+\gamma + \cdots + \gamma^{s-1})^\circ.$
	Then, $m=m_1+m_2$ by considering  the induced action on the tangent space of $\A$ at the  origin.
	We have  $\gamma(P)=P$  for each point  $P$ belonging to the intersection of 
	$\ker (id - \gamma)^\circ $  and $  \ker (id+\gamma + \cdots + \gamma^{s-1})^\circ $, so 
	$0= (id+\gamma + \cdots + \gamma^{s-1})(P)=sP,$ which implies that
	$\ker (id - \gamma)^\circ \cap  \ker (id+\gamma + \cdots + \gamma^{s-1})^\circ  \subseteq \A [s].$
	Thus  $\A$ is $k$-isogenous to their product, i.e.,  
	$$\A \sim_k \ker (id - \gamma)^\circ \times  \ker (id+\gamma + \cdots + \gamma^{s-1})^\circ. $$
	Moreover, we note that $\text{Im} (id+\gamma + \cdots + \gamma^{s-1}) \subseteq \ker(id-\gamma)^\circ$ and 
	$$m-m_2=\dim \text{Im} (id+\gamma + \cdots + \gamma^{s-1}) = \dim \ker(id-\gamma)^\circ =m_1.$$
	Therefore, $\text{Im} (id+\gamma + \cdots + \gamma^{s-1}) =\ker(id-\gamma)^\circ$ that gives the desired result.
\end{proof}

Here, we describe a general method of construction of new $s$-cover using the 
given ones, which we will use  in  the  proof of Theorems \ref{main1} and \ref{main2}.

Let   $\pi_i: \V'_i \rightarrow \V_i$ for $i=1,2$ be $s$-covers of  irreducible quasi-projective varieties, 
all defined over $k$.
Assume  there exist $k$-rational simple points $v'_i \in \V'_i$.  Denote by  
$G_i$   the cyclic Galois group of the corresponding function field extensions.
Then, the covering $\pi_1 \times \pi_2 : \V'_1 \times \V'_2 \rightarrow  \V_1 \times \V_2 $ has  Galois group 
$ G_1 \times G_2 \cong \Z/s\Z \times \Z/s\Z.$ 
Suppose that  $\W$ is   its intermediate cover $\V'_1 \times \V'_2/G$, where 
$G$ is the  group generated by $\gamma= (\gamma_1, \gamma_2) \in \aut(\V'_1 \times \V'_2) $.
Let $\tilde{\gamma}=(\tilde{\gamma}_1, \tilde{\gamma}_2)$  be the order $s$ 
automorphism in $ \aut (\alb (\V'_1) \times \alb(\V'_2))$ corresponding to $\gamma$, where
$\tilde{\gamma}_i$ is an automorphism of $\alb(\V'_i)$
induced by $\gamma_i \in \aut(\V'_i)$ of order $s\geq 2$  for $i=1,2$.
Then there exists a $k$-rational isomorphism 
$$\phi:= \alb(\V'_1)  \times \alb(\V'_2)  \rightarrow \alb(\V'_1\times \V'_2), $$
given by $\phi=\tilde{\phi}_1 + \tilde{\phi}_2$, where $\tilde{\phi}_i: \alb(\V'_i) \rightarrow \alb(\V'_1)  \times \alb(\V'_2)$ is induced  by the inclusion map $\phi_i: \V'_i  \rightarrow \V'_1\times \V'_2 $ defined by
$\phi_1(v)=(v, v'_2)$ and $\phi_2(v)=( v'_1, v)$.
By this isomorphism,   we have 
$\ker( \mu )  \sim_k   \ker(\mu_1) \times \ker(\mu_2),$
where $\mu :=id+ \tilde{\gamma}+ \cdots + \tilde{\gamma}^{s-1}$ and 
$\mu_i :=id+ \tilde{\gamma}_i+ \cdots +\tilde{\gamma}_i^{s-1}$ for $i=1,2$.
This implies that
$\ker( \mu )^\circ  \sim_k   \ker(\mu_1)^\circ \times \ker(\mu_2)^\circ.$
Therefore, applying Lemma (\ref{alblem}) and putting everything together, we conclude the following result.

\begin{prop}
	\label{hhprop1}
	As a $k$-rational isogeny of Abelian varieties, we have
	$$\prym_{\V'_1 \times \V'_2/\W} \sim_k  \prym_{\V'_1/\V_1}\times  \prym_{ \V'_2/\V_2}.$$
\end{prop}

\section{The result of Hazama}
\label{hazres}

In   \cite{haz1, haz3},    Hazama  gave an  explicit method of construction  of Abelian varieties 
that have  large  rank over function fields using the twist theory \cite{bor-ser, haz1}.
In \cite{wbwang}, Wang extended the result of \cite{haz1} to cyclic covers of the projective 
line with prime degrees.
Inspired by Hazama's result, in \cite{yam2}, Yamagishi   reduces the problem of 
constructing elliptic curves of rank $ n \geq 1$ with generators to the problem of finding rational points on a certain varieties. 
By providing a parametrization for  the rational points on those varieties, she  gets all of the 
elliptic curves of  rank  $1 \leq n \leq 7$ defined over a field of characteristic different from two.

Here, we briefly recall the main result of  Hazama  from \cite{haz3}.
Let $\A$ be an Abelian variety  over $k$ with  characteristic different from two.
Suppose that $\pi : \V' \rightarrow \V$ is  a double cover  with Prym variety 
$\prym (\V'/\V)$, of irreducible  quasi-projective varieties   $\V$ and $\V'$ defined over $k$.
Let  $\KK$ and   $\KK'$  be the function field of $\V$ and $\V'$, respectively, and 
$G$  the Galois group of the extension $\KK'|\KK$. Denote by  $\A_a$  the twist of $\A$ by 
the $1$-cocycle   $a=(a_u) \in Z^1(G, Aut(\A))$ defined by $a_{id}=id$ and $a_{\iota}=-id$.

\begin{thm}
	\label{fhaz3}
	With the above notations, assume that there exist a $k$-rational simple point $v'_0 \in \V'$.  
	Then we have an isomorphism of  Abelian groups:
	$$\A_a(\KK) \cong \Hm_{k} (\prym (\V'/\V), \A) \oplus \A [2](k).$$
	Moreover,  if  $\prym_{\V'/\V}$ is $k$-isogenous with $ \Ee^n \times \B$ for some positive integer $n$,
	where $\Ee$ is an  elliptic curve  over $k$ and $\B$ is an Abelian variety  with  no simple component
	$k$-isogenous to $\Ee$, then   
	$\rk(\Ee_b(\KK))=n\cdot \rk (\enn_{k}(\Ee)).$
\end{thm}

We refer the reader to   2.2  and  2.3 in \cite{haz3}, for the proof of the above theorem.
%
\section{Proof of Theorem \ref{main1}}
\label{mproof1}
Suppose that the natural map $i_{\V'} :  \V' \rightarrow \alb (\V')$ sends $v'_0$ to the origin  
of $\alb(\V')$ so that $i_{\V'}$ is defined over $k$. Then, using  Theorem 4 of  chapter II in \cite{lang3}, we have 
$\A (\KK')=\{ \text{$k$-rational maps $\V' \rightarrow \A$ }\} \cong  \Hm_k(\alb(\V'), \A) \oplus \A (k),$
where $P \in \A (\KK')$ corresponds to the pair $(\lambda, Q) \in  \Hm_k(\alb(\V'), \A) \oplus \A (k)$ such that 
$P(v')= \lambda (i_{\V'}(v'))+Q$  for each $v' \in \V'$. 
This implies that the action of $\gamma^j \in G$ is given by  $\gamma^j (\lambda, Q)= (\lambda \circ \tilde{\gamma}^j, Q)$ 
for $j=0, \cdots , s-1,$ where $\tilde{\gamma}$ is the automorphism of the Albanese variety $\alb(\V')$ induced by $\gamma \in \aut(\V').$
Since $\gamma^s=id$ and hence $\tilde{\gamma}^s=id$, so 
$$\A_a (\KK)\cong \{ P \in \A (\KK'): b_{\gamma^j} \cdot  ^{\gamma^j}(P)=P, \ \forall {\gamma^j} \in G\},$$ 
by applying the proposition 1.1 in \cite{haz1}. 
This implies that   $(\lambda, Q) \in \A_a(\KK)$ if and only if 
$ \gamma^j(\lambda , Q)=(\lambda \circ \tilde{\gamma}^j, Q)= (\lambda \circ \tilde{\gamma}^{s-j}, Q)=\gamma^{s-j}(\lambda , Q).$
Thus,  $(\lambda, Q) \in \A_a(\KK)$ if and only if 
$\lambda$ annihilates  $Im (id+ \tilde{\gamma} + \cdots+ \tilde{\gamma}^{s-1})$ and $Q \in \A [s](k)$. 
Therefore, 
$$\A_a(\KK) \cong \Hm_{k} (\prym_{\V'/\V}, \A) \oplus \A [s](k).$$

Furthermore, if we assume that $\prym_{\V'/\V}$ is $k$-isogenous with $ \A^n \times \B$ for some positive integer $n$,
where $\A$ and $\B$ are   Abelian varieties  defined over $k$ such that $\dim(\B)=0$  or $\dim(\B)> \dim (\A)$
and  none of  irreducible components of $\B$ is $k$-isogenous to $\A$, then
\begin{align*}
\A_a(\KK) & \cong  \Hm_{k} (\prym_{\V'/\V}, \A) \oplus \A [s](k)\\
& \cong  \Hm_{k} (\A^n \times \B, \A) \oplus \A [s](k)\\
& \cong  \Hm_{k} (\A^n, \A) \oplus \Hm_{k} (\B, \A) \oplus \A [s](k)\\
& \cong (End_{k}(\A))^n   \oplus \Hm_{k} (\B, \A) \oplus \A [s](k).
\end{align*}
Therefore, as $\Z$-modules, we have 
$\rk(\A_a(\KK))\geq n\cdot \rk (\enn_{k}(\A)).$
\section{The proof of the theorem \ref{main2}}
\label{rpavff}
Given the integers $2 \leq s \leq r \leq  n$,   fix a  polynomial  $f(x)\in k[x]$ of degree $r$. 
Consider the curve $\Cc_{s, f}: y^s=f (x)$ with a rational point $c \in \Cc_{s,f}(k)$.
It admits an order $s$ automorphism $ \iota: (x,y) \mapsto (x, \zeta \cdot y)$.  For each $1\leq i \leq n$,  let $\Cc_{s, f}^{(i)}$
a copy of $\Cc_{s, f}$ defined by  the affine equation  $y_i^s=f(x_i)$ and  
denote by  $\iota_i$ the corresponding automorphism for each of these curves.
Define  $\Cb_n:= \prod_{i=1}^n \Cc_{s, f}^{(i)}$ which can be expressed by  the simultaneous equations
$y_i^s=f(x_i)$ for $i=1, \cdots, n$.  Let  $G=\langle \gamma \rangle$ to be  the order $s$ cyclic subgroup 
of $ \aut(\Cb_n)$, where  $\gamma:=( \iota_1, \cdots, \iota_n)$, and  define $\Vb_n:=\Cb_n/G$. 
If we assume that  $L$ is the function field of $\Cb_n$, i.e., $L=k( x_1, x_2, \cdots ,x_n, y_1, y_2, \cdots,   y_n),$ 
where  $x_1, x_2, \cdots , x_n$ are independent transcendentals variables and each $y_i$ 
defines a degree $s$ extension by  the equation $y_i^s-f(x_i)=0$, 
then $K=k(\Vb_n)$ the function field of $\Vb_n$ is equal to the set of all invariant elements of $L$ by the action of $G$, i.e.,
$K=L^G=k ( x_1,\cdots, x_n, y_1^{s-1}y_2, \cdots, y_1^{s-1}y_{n-1}).$
Since $(y_1^{s-1}y_{i+1})^s=f(x_1)^{s-1}f(x_{i+1})$ holds for $i=1, \cdots, n-1$, so by defining   $z_i:=y_1^{s-1}y_{i+1}$  
the variety  $\Vb_m$ can be expressed by  $z_i^s=f(x_1)^{s-1}f(x_{i+1})$, for $i=1,\cdots, n-1$.
Hence,  $L|K$ is  a   cyclic extension of degree $s$ determined by   $y_1^s=f(x_1)$, i.e., 
$$L=K(y_1)=k(x_1, \cdots, x_n, z_1, \cdots, z_{n-1})(y_1).$$
Define $\Cc_{s, f}^{\xi}$ to be the twist  of $\Cc_{s, f}$  by the extension $L|K$ 
and let $J(\Cc_{s, f}^{\xi})$ to denote  its Jacobian variety.
In a similar way as  Corollary 3.1 in \cite{haz1}, one can see that $\Cc_{s, f}^{\xi}$ 
is defined by the affine equation   $f(x_1)z^s=f(x).$
It is also easy to check that  $\Cc_{s, f}^{\xi}$ contains   the  following $K$-rational points:
\begin{equation}
\label{Kpoint}  
P_1:=(x_1, 1)\ \text{and} \ P_i:=(x_{i+1}, z_i/f(x_1)) \ \text{ for} \ ( 1 \leq i \leq n-1).
\end{equation}

\begin{rem}
	The  construction of the varieties $\Cb_n$ and $\Vb_n$ generalizes that  given by Yamagishi in  \cite{yam2},
	which is used to find elliptic curves of high  rank  that have a  given set of algebraic numbers   as $x$-coordinates of
	the  generators of their Mordell-Weil group.
\end{rem}
By the fact that the Albanese and Jacobian varieties of curves coincide and
applying  Lemma \ref{alblem}  to  $\V'=\Cc_{s,f}^{(i)}=\Cc_{s,f}$ and $\V=\Pp^1$, we have 
$$\prym_{\Cc_{s,f}^{(i)}/\Pp^1} =\frac{J(\Cc_{s,f}^{(i)})}{\text{Im} (id+ \tilde{\iota}+ \cdots +\tilde{\iota}^{s-1})}
\sim_k \ker \big(id+ \tilde{\iota}+ \cdots +\tilde{\iota}^{s-1})^\circ.$$
Since $0=id - \tilde{\iota}^s=(id-\tilde{\iota})(id+ \tilde{\iota}+ \cdots +\tilde{\iota}^{s-1})$ 
and $id \not = \tilde{\iota}$,   we have 
$$0=id+ \tilde{\iota}+ \cdots +\tilde{\iota}^{s-1} \in End(J(\Cc_{s,f}^{(i)}))=End(J(\Cc_{s,f})),$$
which implies that $\prym_{\Cc_{s,f}^{(i)}/\Pp^1}=J(\Cc_{s,f}^{(i)})$ for each $i=1,\cdots, n.$ 
By applying Proposition  \ref{hhprop1}, one can get an $k$-isogeny of Abelian varieties
\begin{equation}
\label{prym}
\prym_{\Cb_n/\Vb_n} \sim_k \prod_{i=1}^n \prym_{\Cc_{s,f}^{(i)}/\Pp^1} = J(\Cc_{s,f})^n.
\end{equation}

Let us   denote by  $Q_i$ the image of $P_i$ ( $i=1,\cdots, n$)
given by  (\ref{Kpoint}) under the canonical   embedding of $\Cc_{s, f}^{\xi} $ into  $ J(\Cc_{s, f}^{\xi})$.
Define $a=(a_u) \in Z^1(G, \aut(J(\Cc_{s,f}))$  by $a_{id}=id$ and $a_{\gamma^j}=\tilde{\iota}^j$ where 
$\gamma^j \in G$ and 
$\tilde{\iota}: J(\Cc_{s,f}) \rightarrow J(\Cc_{s,f})$ is the automorphism induced by $\iota: \Cc_{s,f} \rightarrow \Cc_{s,f}$.
Denote by $J(\Cc_{s,f})_a$ the twist of $J(\Cc_{s,f})$ with the $1$-cocycle $a$. Then, 
$J(\Cc_{s,f})_a=J(\Cc_{s,f}^\xi)$ by  the lemma on page 172 in  \cite{haz1}.
Applying the  theorem \ref{main1}  for $\V'=\Cb_n$,  $\V=\Vb_n$, and $\A =J(\Cc_{s,f})$, we have  
\begin{align*}
J(\Cc_{s,f}^\xi)(K) & \cong  \Hm_{k} (\prym_{\Cb_n/\Vb_n}, J(\Cc_{s,f})) \oplus J(\Cc_{s,f})[s](k)\\
& \cong  \Hm_{k} (J(\Cc_{s,f})^n , J(\Cc_{s,f})) \oplus J(\Cc_{s,f})[s](k)\\
& \cong (\enn_{k}(J(\Cc_{s,f})))^n    \oplus J(\Cc_{s,f})[s](k).
\end{align*}
Thus, as $\Z$-modules, we have $\rk(J(\Cc_{s,f}^\xi)(K)\ge n\cdot \rk (\enn_{k}(J(\Cc_{s,f}))).$
Tracing back the above isomorphisms  shows that the points $Q_1, \cdots, Q_n$ belong to 
the set of independent generators of  $J(\Cc_{s,f}^\xi)(K)$. 




\begin{thebibliography}{99}                                                                                                %
																									%
\bibitem {beavi}   \textsc{Beauville, A.}: \  \textit{Vari\'{e}t\'{e}sde Prym et Jacobiennes intermediares}, 
Ann. scient. \'{E}c.  Norm. Sup.", \textbf{10}, (1977), 309-391.
																																																	%
\bibitem {bor-ser}   \textsc{Borel, A.,  and  Serre, J. -P. }: \  \textit{Th\'{e}or\`{e}mes de finitude en cohomologie galoisienne}, 
Comment. Math. Helv., \textbf{39}, (1964), 111-164.


\bibitem {duje1} \textsc{ Dujella, A.}:\ \textit{High rank elliptic curves with prescribed torsion}, 
{\ttfamily  http://www.maths.hr/~duje/tors.htl, } (2017)








\bibitem {haz1}   \textsc{Hazama, F.}: \  \textit{On the Mordell-Weil group of   certain abelian varieties defined over function fields}, 
J. Number Theory, \textbf{37}, (1991), 168-172.

\bibitem {haz3}   \textsc{Hazama, F.}: \  \textit{Rational points on certain abelian varieties over  function fields}, 
J. Number Theory, \textbf{50}, (1995), 278-285.

\bibitem {slm3} \textsc{Hindry, H., and Silverman,  J. H.}:\ \textit{Diophantine geometry: An introduction},
Graduate Text in Mathematics, Vol. \textbf{201}, Springer-Verlag, New york (2001).


\bibitem {lang3} \textsc{Lang, S.}:\ \textit{Abelian Varieties}, Springer-Verlag, New York/Berlin (1983).


\bibitem {ortega} \textsc{Lange, H.,  and   Ortega, A.}:\ \textit{Prym varieties of cyclic coverings}, 
Geom. Dedicata., \textbf{150}, (2011), 391-403.


\bibitem {mumf} \textsc{Mumford, D.}:\ \textit{Prym varieties (I)}, Contributions to analysis
 (a collection of papers dedicated to Lipman Bers), Academic Press, New York, (1974), 325-350.



\bibitem {poonen} \textsc{Park  J., and et. al.}:\ \textit{A heuristic for boundedness of ranks of elliptic curves}, 
{\ttfamily  https://arxiv.org/abs/1602.01431, } (2016)



\bibitem {salami0} \textsc{Salami, S.}:\ \textit{On some different related problems in Diophantine geometry}, 
Ph.D thesis, Universidade Federal do Rio de Janeiro,  Brazil, (2017).


\bibitem {slm1} \textsc{Silverman,  J. H.}:\ \textit{The Arithmetic of Elliptic Curves}, second edition, 
Graduate Text in Mathematics, Vol. \textbf{106}, Springer-Verlag, New york (2009).



\bibitem {tate-shaf}   \textsc{Tate, J. T. and Shafarevich, I. R.}: \  \textit{The rank of elliptic curves}, 
Akad. Nauk SSSR, \textbf{175}, (1967), 770-773.

\bibitem {ulmer1}  \textsc{Ulmer, D.}: \  \textit{Elliptic Curves with Large rank over function fields}, 
Annals  Math., \textbf{155}, No. 1 (2002), 295-315.

\bibitem {ulmer2}  \textsc{Ulmer, D.}: \  \textit{Rational curves on elliptic surfaces}, 
J. Algebraic Geom., \textbf{26},  (2017), 357-377.


\bibitem {wbwang}   \textsc{ Wang, W. B.}: \  \textit{On the Twist of Abelian Varieties Defined by the Galois Extension of Prime Degree}, Journal of Algebra, \textbf{163}  (3),  (1994), 813 - 818.




\bibitem {yam2}   \textsc{Yamagishi, H.}: \  \textit{A unified method of construction of elliptic curves with high Mordell-Weil rank}, 
Pacific J. Math, \textbf{191}, (1999), 507-524.




\end{thebibliography}
\end{document}